\documentclass[a4paper,10pt]{amsart}
\usepackage{amstext}
\usepackage{amsbsy}
\usepackage{amsopn}
\usepackage{amsfonts}
\usepackage{amsmath}
\usepackage{amssymb}

\renewcommand{\phi}{\varphi}
\newcommand{\Ge}{\mathfrak{g}}
\newcommand{\Ges}{\mathfrak{g}^{\star}}

\newcommand{\Fe}{\mathcal{F}}
\newcommand{\e}{\epsilon}
\newcommand{\R}{\boldsymbol{R}}
\newcommand{\eR}{\mathcal{R}}

\newcommand{\Schw}{\mathcal{S}}

\newcommand{\C}{\boldsymbol{C}}

\swapnumbers
\newtheorem{theorem}[equation]{Theorem} 
\newtheorem{lemma}[equation]{Lemma}    
\newtheorem{corollary}[equation]{Corollary}
\newtheorem{proposition}[equation]{Proposition}
\theoremstyle{remark}
\newtheorem{remark}[equation]{Remark}
\numberwithin{equation}{section}

\title[Flag kernels]
 {$\bf L^p$-boundedness of flag kernels\\ on homogeneous groups} 

\author{P. G{\l}owacki}
\address{Institute of Mathematics, University of Wroc{\l}aw,
pl. Grunwaldzki 2/4, 50-384 Wroc{\l}aw, Poland}
\subjclass[2000]{43A15 (primary), 43A32, 43A85 (secondary).}
\begin{document}

\begin{abstract}
We prove that  the flag kernel singular integral operators of  Nagel-Ricci-Stein on a homogeneous group are bounded on  $L^p$, $1<p<\infty$. The gradation associated with the kernels is the natural gradation of the underlying Lie algebra. Our main tools are the Littlewood-Paley theory and a symbolic calculus combined in the spirit of Duoandikoetxea and Rubio de Francia.
\end{abstract}

\maketitle
\section{Introduction}
Flag kernels on homogeneous groups have been introduced by Nagel-Ricci-Stein \cite{nagel} in their study of quadratic $CR$-manifolds. They can be regarded as a generalization of Calder\'on-Zygmund singular kernels with singularities extending over the whole of the hyperspace $x_1=0$, where $x_1$ is the top level variable. The definition is complex (see below), as it involves cancellation conditions for each variable separately. However, the descritption of flag kernels in terms of their Fourier transforms is much simpler and bears a striking resemblance to that of the symbols of convolution operators considered independently by  the author (in, e.g. \cite{arkiv2007}).

In Nagel-Ricci-Stein \cite{nagel} we find an $L^p$-boundedness theorem for the very special flag kernels where the associated gradation consists of commuting subalgebras of the underlying Lie algebra of the homogeneous group. The natural question of what happens if the gradation is the natural gradation of the homogeneous Lie algebra is left open. The aim of this paper is to answer the question in the affirmative. We prove that such flag kernels give rise to bounded operators.

The smooth  symbolic calculus mentioned above has been adapted to an extended class of flag kernels of small (positive and negative) orders and combined with a variant of the Littlewood-Paley theory built on a stable semigroup of measures with smooth densities very similar to the Poisson kernel on the Euclidean space. The strong maximal function of Christ \cite{christ} is also instrumental. The  approach has been inspired by the well-known paper by Duoandicoetxea and Rubio de Francia \cite{duoandi}. The dependence of the present paper on Duoandicoetxea and Rubio de Francia \cite{duoandi} is evident throughout.

The class of flag kernels dealt with here is in fact an algebra. For this the reader is referred to  \cite{colloquium2010} where also the $L^2$-boundedness of flag kernels is proved solely by means of the symbolic calculus. 

After this paper had been completed,  a preprint of Nagel-Ricci-Stein-Wainger \textit{Singular integrals with flag kernels on homogeneous groups I,} has been made available, where the $L^p$-boundedness theorem for flag kernels is proved. This comprehensive treatment of flag kernels on homogeneous groups has been announced for some time. Professor Stein has lectured a couple of times on the subject, see, e.g. \cite{stein}. The authors also use a version of Littlewood-Paley theory but otherwise the approach differs from the one presented here in many respects, the most important being our use of the symbolic calculus and partitions of unity related to a stable semigroup of measures. That is why we believe that what  is presented here has an independent  value and may count as a contribution to the theory.   

\section{Preliminaries}
Let $\Ge$ be a nilpotent Lie algebra with a fixed Euclidean structure and $\Ges$ its dual.  Let $\delta_tx=tx$, $t>0$ be a family of  dilations on $\Ge$ and let
\[
\Ge_j=\{x\in\Ge: \delta_tx=t^{p_j}\cdot x\},
\hspace{2em}
1\le j\le d,
\]
where $1=p_1< p_2<\dots <p_d$. Denote by
\[
Q_j=p_j\cdot\dim\Ge_j
\]
the homogenous dimension of $\Ge_j$. The homogeneous dimension of $\Ge$ is
\[
Q=\sum_{j=1}^dQ_j.
\]

We have
\begin{equation}\label{grad}
\Ge=\bigoplus_{j=1}^d\Ge_j,
\qquad
\Ges=\bigoplus_{j=1}^d\Ges_j
\end{equation}
and
\[
[\Ge_i,\Ge_j]\subset \left\{
\begin{array}{ll}
\Ge_k, & {\rm if \ }  p_i+p_j=p_k,\cr \{0\}, & {\rm if \ } p_i+p_j\notin{\mathcal{P},}
\end{array}
\right.
\]
where $\mathcal{P}=\{p_j:1\le j\le d\}$.

Let
\[
x\to|x|\approx\sum_{j=1}^d\|x_j\|^{1/p_j}
\]
be a homogeneous norm on $\Ge$ smooth away from the origin. Let also
\[
|x|_j=|(x_1,x_2,\dots,x_j,0,\dots,0)|,
\qquad
1\le j\le d.
\]
In particular, $|x|_1=|x_1|$, and $|x|_d=|x|$. Another notation will be applied to $\Ges$. For $\xi\in\Ges$,
\[
|\xi|_j=|(0,\dots,0,\xi_j,\xi_{j+1},\dots,\xi_d)|,
\qquad
1\le j\le d.
\]
In particular, $|\xi|_1=|\xi|$, and $|\xi|_d=|\xi_d|$.

We shall also regard $\Ge$ as a Lie group with the
Campbell-Hausdorff multiplication
\[
xy=x+y+r(x,y),
\]
where $r(x,y)$ is the (finite) sum of terms of order at least $2$ in the Campbell-Hausdorff
series for $\Ge$. Under this identification the homogeneous ideals
\[
 \Ge^{(k)}=\bigoplus_{j=k}^d\Ge_j
\]
are normal subgroups.

In expressions like $D^{\alpha}$ or $x^{\alpha}$ we shall use multiindices
\[
\alpha=(\alpha_1,\alpha_2,\dots,\alpha_d),
\]
where 
\[
\alpha_k=(\alpha_{k1},\alpha_{k1},\dots,\alpha_{kn_k}),
\qquad
n_k=\dim\Ge_k=\dim\Ges_k,
\] 
are themselves multiindices with positive integer entries corresponding to the spaces $\Ge_k$ or $\Ges_k$. The homogeneous length of $\alpha$ is defined by
\[
|\alpha|=\sum_{k=1}^d|\alpha_k|,
\qquad|\alpha_k|=p_k(\alpha_{k1}+\alpha_{k2}+\dots+\alpha_{kn_k}). 
\]

The Schwartz space of smooth functiions which vanish rapidly at infinity along with their derivatives will be denoted by $\Schw(\Ge)$. For a tempered distribution $K$, that is a continuous linear functional on $\Schw(\Ge)$, we shall write
\[
 \langle K,f\rangle=\int_{\Ge}f(x)K(x)\,dx,
\qquad
f\in\Schw(\Ge),
\]
without implying thereby that $K$ is a locally integrable function.

 Even though the flag kernels are our prime concern here we need a broader class of kernels to properly deal with them. In \cite{studia2010}, we proposed a natural generalization of  the flag kernels of Nagel-Ricci-Stein. Let 
\[
 \|f\|_{(k)}=\max_{|\alpha|\le Q_k+1}\sup_{x\in\Ge_k}(1+|x|)^{Q_k+1}|D^{\alpha}f(x)|
\]
be a fixed norm in the Schwartz space $\Schw(\Ge_k)$. Let
\[
\mathcal{N}=\{\nu=(\nu_1,\nu_2,\dots,\nu_d): |\nu_k|<Q_k, \ 1\le k\le d\}.
\]
Let $\nu\in\mathcal{N}$. We define the class $\Fe(\nu)$ by induction on the homogeneous step $d$. When $d=0$ the elements of $\Fe(\emptyset)$ are simply constants. If $d\ge1$, we say that a distribution $K\in\Schw^{\star}(\Ge)$ is in $\Fe(\nu)$ if it is smooth away from the hyperspace $x_1=0$ and satisfies the following conditions:

i) For every multiindex $\alpha$,
\begin{equation}\label{size}
|D^{\alpha}K(x)|\le C_{\alpha}|x|_1^{-\nu_1-Q_1-|\alpha_1|}|x|_2^{-\nu_2-Q_2-|\alpha_2|}\dots
|x|_d^{-\nu_d-Q_d-|\alpha_d|}
\end{equation}
for $x_1\neq0$; 

ii) For any $1\le k\le d$,  
\begin{equation}\label{cancellation}
<K_{R,\phi},f>=R^{-\nu_k}\int_{\Ge}\phi(Rx_k)f(x_1,\dots,x_{k-1},x_{k+1},\dots,x_d)K(x)\,dx
\end{equation}
is in $\Fe(\nu_{(k)})$ on $\oplus_{j\neq k}\Ge_j$, where $\nu_{(k)}=(\nu_1,\dots,\nu_{k-1},\nu_{k+1},\dots,\nu_d)$, and this is uniform in $\phi\in\Schw(\Ge_1)$ with $|\phi\|_{(k)}\le1$ and $R>0$.  (Note that the meaning  of \textit{uniform boundedness} of a family of members of $\Fe(\nu)$ is obvious in the case $d=0$ and, for $d\ge1$, can be defined by induction.)

 For every $N$, we define a norm $\|\cdot\|_{\nu,N}$ in $\Fe(\nu)$ as the maximum of all the bounds occurring in the definition. First, we let
\[
 s_N^{\nu}(P)=\max_{|\alpha|\le N}\sup_{x_1\neq0}\prod_{k=1}^d|x|_k^{Q_k+\nu_k+|\alpha_k|}|D^{\alpha}K(x)|.
\]
and, if $d=1$, 
\[
 \|K\|_{\nu_1,N}=s_N^{\nu_1}(K)+\sup_{|\phi\|_{(1)}\le1}\sup_{R>0}R^{-\nu_1}|<K,\phi\circ\delta_R>|.
\]

If $d>1$, we let
\[
 \|K\|_{\nu,N}=s_N^{\nu}(K)+\max_{1\le k\le d}\sup_{\|\phi\|_{(k)}\le1}\sup_{R>0}\|K_{R,\phi}\|_{\nu_{(k)},N}.
\]
Thus, $\Fe(\nu)$ can be regarded as a locally convex topological vector space. Let us remark that  $\Fe(0)=\Fe(0,0,\dots,0)$ is exactly the class of flag kernels of Nagel-Ricci-Stein \cite{nagel} (see Corollary 3.7 of \cite{studia2010}).

For a $K\in\Schw^{\star}(\Ge)$, let
\[
<\widetilde{K},f>=\int_{\Ge}f(x^{-1})K(dx),
\qquad
f\in\Schw(\Ge).
\]
The following three propositions have been proved in \cite{colloquium2010} and \cite{studia2010}.

\begin{proposition}[(Theorem 2.5 of \cite{colloquium2010})]\label{l2bound_flag}
 Let $K\in\Fe(0)$ be a flag kernel on $\Ge$. The convolution operator
$f\to f\star\widetilde{K}$
defined initially on $\Schw(\Ge)$ extends uniquely to a bounded operator on $L^2(\Ge)$. \end{proposition}

\begin{proposition}[(Proposition 1.5 of \cite{studia2010})]\label{gen_flag}
Let $\nu\in\mathcal{N}$. A distribution $K$ is in $\Fe(\nu)$ if and only if its Fourier transform is locally integrable, smooth for $\xi_d\neq0$, and satisfies
\begin{equation}\label{multiplier}
|D^{\alpha}\widehat{K}(\xi)|\le C_{\alpha}|\xi|_1^{\nu_1-|\alpha_1|}\dots|\xi|_d^{\nu_d-|\alpha_d|},
\qquad
\xi_d\neq0.
\end{equation}
\end{proposition} 
\noindent
Cf. also the original Theorem 2.3.9 of Nagel-Ricci-Stein \cite{nagel} for kernels $K\in\Fe(0)$.
\begin{proposition}[Theorem 4.8 of \cite{studia2010}]\label{composition_flag}
Let $\nu,\mu,\nu+\mu\in\mathcal{N}$. Let $K\in\Fe(\nu)$, $L\in\Fe(\mu)$. Let $\phi=\otimes_{k=1}^d\phi_k\in C_c^{\infty}(\Ge)$ be equal to $1$ in a neighbourhood of $0$. There exists a $P=P_{K,L}\in\Fe(\nu+\mu)$ such that
\[
P=\lim_{\e\to0}K_{\e}\star L
\]
in the sense of distributions, where 
\[
<K_{\e},f>=\int_{\Ge}\phi(\e x)f(x)K(dx),
\qquad
f\in\Schw(\Ge).
\]
Moreover, the mapping $(K,L)\to P_{K,L}$ is continuous.
 \end{proposition}

\section{Semigroups of measures}

Following Folland-Stein \cite{folland}, we say that a function $\phi$ belongs to the class $\mathcal{R}(a)$, where $a>0$, if it is smooth and
\begin{equation}\label{rclass}
|D^{\alpha}\phi(x)|\le C_{\alpha}(1+|x|)^{-Q-a-|\alpha|},
\qquad
{\rm all \ } \alpha.
\end{equation}

\begin{proposition}\label{fourier}
Let $\phi\in\mathcal{R}(a)$ for some $0<a<1$ and let $\int\phi=0$. Then $\phi\in\Fe(a)$.
\end{proposition}
\begin{proof}
The size condition (\ref{size}) follows by (\ref{rclass}). To verify the cancellation condition (\ref{cancellation}) let $f\in\Schw(\Ge)$ and $R>0$. Then
\begin{align*}
\int_{\Ge}&f(Rx)\phi(x)\,dx=\int_{\Ge}\big(f(Rx)-f(0)\big)\phi(x)\,dx \\
&\le\int_{|x|\le R^{-1}}\big(f(Rx)-f(0)\big)\phi(x)\,dx
+\int_{|x|\ge R^{-1}}\big(f(Rx)-f(0)\big)\phi(x)\,dx \\
&\le\|f\|\left(R\int_{|x|\le R^{-1}}|x|^{-Q-a+1}\,dx+2\int_{|x|\ge R^{-1}}|x|^{-Q-a}\,dx\right)\\
&\le CR^a\|f\|,
\end{align*}
where $\|\cdot\|$ is a Schwartz class norm. 
\end{proof}

Let
\[
\langle P,f\rangle
=\lim_{\e\to}\int_{|x|\ge\e}\Big(f(0)-f(x)\Big)\frac{dx}{|x|^{Q+1}},
\qquad
f\in\Schw(\Ge).
\]
The distribution $P$ is an infinitesimal generator of a continuous semigroup of probability measures with smooth densities
\[
h_t(x)=t^{-Q}h(t^{-1}x),
\]
where $h\in\eR(1)$ and $P^Nh\in\eR(N)$ for $N=1,2,\dots$. In other words,
\[
h_t\star h_s=h_{t+s},
\qquad
t,s>0. 
\]
and
\[
\frac{d}{dt}\big|_{t=0}<h_t,f>=-<P,f>,
\qquad
f\in\Schw(\Ge),
\]
The operator $Pf=f\star P$ is essentially selfadjoint with $\Schw(\Ge)$ for its core domain. The reader is referred to \cite{inventiones1986} for proofs and details. 

For $0<{a}<1$
\begin{equation}\label{m}
\langle P^{a},f\rangle=\frac{1}{\Gamma(-{a})}\int_0^{\infty}t^{-1-{a}}\langle
\delta_0-h_t, f\rangle\,dt=\frac{1}{\Gamma(1-{a})}\int_0^{\infty}t^{-{a}}\langle
Ph_t, f\rangle\,dt
\end{equation}
defines a homogeneous distribution smooth away from the origin (cf., e.g. Yosida \cite{yosida}).

\begin{proposition}\label{pm}
For every $0<a<1$,
\[
P^{a}h\in\eR(a)
\qquad
{\rm and}
\qquad
\int_{\Ge}P^{a}h(x)\, dx=0.
\]
\end{proposition}

\begin{proof}
By (\ref{m}),
\[
P^{a}h(x)=\frac{1}{\Gamma(1-{a})}\int_0^{\infty}t^{-{a}}Ph_{t+1}(x)\,dt,
\]
whence
\begin{align*}
|D^{\alpha}P^{a}h(x)|&\le\frac{C_{\alpha}}{\Gamma(1-{a})}\int_0^{\infty}\frac{t^{-{a}}\,dt}{(t+1+|x|)^{Q+1+|\alpha|}}\\
&\le C_{\alpha}'\int_0^{\infty}\frac{t^{-{a}}\,dt}{(\frac{t}{1+|x|}+1)^{Q+1+|\alpha|}}\cdot(1+|x|)^{-Q-1-|\alpha|}\\
&\le C_{\alpha}''\int_0^{\infty}\frac{t^{-{a}}\,dt}{(t+1)^{Q+1+|\alpha|}}\cdot(1+|x|)^{-Q-{a}-|\alpha|},
\end{align*}
as required.

Now, for every $t>0$,
\[
\int h_t\,dx=1.
\]
Therefore,
\[
\int Ph_t\,dx=-\frac{d}{dt}\int h_t\,dx=0,
\qquad
t>0.
\]
which combined with (\ref{m}) gives the second part of the assertion. 
\end{proof}

\section{Littlewood-Paley theory}
From now on we fix the function $\phi=P^{1/2}h_{1/2}$.  
\begin{remark}\label{fixfi}
By the results of the previous section, $\phi$ is a smooth function satisfying the estimates
\begin{equation}\label{fiestim}
|D^{\alpha}\phi(x)|\le C_{\alpha}(1+|x|)^{-Q-1/2-|\alpha|}.
\end{equation}
Moreover, $\phi\in\Fe(1/2)$.
\end{remark}
\begin{lemma}\label{fi}
We have
\[
f=\int_0^{\infty}f\star\phi_t\star\phi_t\,\frac{dt}{t},
\qquad
f\in\Schw(\Ge).
\]
\end{lemma}
\begin{proof}
By the semigroup properties,
\[
 -\frac{d}{dt}f\star h_t=f\star Ph_t=\frac{1}{t}f\star(\phi_t\star\phi_t),
\]
whence
\[
\int_{\e}^Mf\star\phi_t\star\phi_t\frac{dt}{t}=f\star h_{\e}- f\star h_M.
\]
Now, if $\e\to0$ and  $M\to\infty$, the expression on the right hand side tends to $f$ in the sense of distributions.
\end{proof}

Let $T=(t_1,\dots,t_d)\in\R_+^d$. We shall regard $\R_+^d$ as a product of copies of the multiplicative group $\R^+$. We shall write
\[
T^{a}=(t_1^{a},\dots, t_d^{a}),
\qquad
TS=(t_1s_1,\dots,t_ds_d),
\qquad
\frac{dT}{T}=\frac{dt_1\dots dt_d}{t_1\dots t_d},
\qquad
a\in\R.
\]

 Let $\phi_k$ be the counterpart of  $\phi$  for $\Ge$ replaced by $\Ge^{(k)}$, $1\le k\le d$. Let
\[
\Phi_k=\delta_k\otimes\phi_k,
\]
where $\delta_k$ stands for the Dirac delta at $0\in\oplus_{j=1}^{k-1}\Ge_j$. Let
\[
\Phi=\Phi_1\star\Phi_2\star\dots\star\Phi_d,
\]
and
\[
\Phi_T=(\Phi_1)_{t_1}\star\dots\star(\Phi_d)_{t_{d}},
\qquad
T\in\R_+^d.
\]
\begin{corollary}\label{unity}
We have
\[
\Phi\in|\Fe|(1/2):=\bigcap_{\e\in\{-1,1\}}\Fe(\e_1/2,\dots,\e_d/2).
\]
Furthermore,
\[
f=\int_{\R_+^d}f\star\Phi_T\star\widetilde{\Phi_T}\,\frac{dT}{T},
\qquad
f\in\Schw(\Ge).
\]
\end{corollary}
\begin{proof}
By  Remark \ref{fixfi},
\[
\Phi_k\in\Fe(0,\dots0,1/2,0,\dots,0)\cap\Fe(0,\dots0,-1/2,0,\dots,0),
\]
where the only nonzero term stands on the $k$-th position. Therefore the  first part of our assertion 
follows by Proposition \ref{composition_flag}. The second one is a consequence of Lemma\nobreak \ \ref{fi}. 
\end{proof}
\begin{proposition}\label{paley}
The Paley-Littlewood square function
\[
G_{\Phi}(f)(x)=\left(\int_{\R_+^d}|f\star\Phi_T(x)|^2\,\frac{dT}{T}\right)^{1/2},
\]
is bounded as an operator on $L^p(\Ge)$. In other words, for every $1<p<\infty$, there is a constant $C_{\phi,p}>0$ such that
\[
\|G_{\Phi}(f)\|_p\le C_{\phi,p}\|f\|_p,
\qquad
f\in\Schw(\Ge).
\]
\end{proposition}
\begin{proof}
The proof is implicitly contained in Folland-Stein \cite{folland} (see Theorem 6.20.b and Theorem 7.7) so we dispense ourselves with presenting all details.

We start with defining some Hilbert spaces and operators. Let $X_0=\C$ and
\[
X_k=L^2(\R^k_+,\frac{dT}{T}),
\qquad
1\le k\le d.
\]
For a given $x\in\Ge$, let $F_k(x):X_{k-1}\to X_k$ be given by
\[
F_k(x)m(t_1,\dots,t_{k-1},t_k)=(\phi_k)_{t_k}(x_k,\dots,x_d)m(t_1,\dots,t_{k-1}).
\qquad
m\in X_{k-1}.
\]
Finally, let $W_k:C_c(\Ge,X_{k-1})\to C_0(\Ge,X_k)$ be the operator
\[
W_kf(x)(T,t_k)=(f\star F_k)(x)(T,t_k)=\int_{\Ge^{(k)}}(\phi_k)_{t_k}(y)f(xy)(T)\,dy,
\]
where $T=(t_1,\dots,t_{k-1})$. Note that $W_k$ acts only on $(x_k,\dots,x_d)$-variable.

We claim that 
$$
W_k:L^2(\Ge,X_{k-1})\to L^2(\Ge,X_k)
$$
is an isometry. In fact, by definition of $\Phi_k$,
\begin{align*}
\|W_kf\|_{L^2(\Ge,X_k}^2&=\int_{\Ge}\|W_kf(x)\|_{X_k}^2\,dx\\
&=\int_{\Ge}dx\int_0^{\infty}\frac{dt}{t}
\int_{\R^{k-1}_+}\frac{dT}{T}\int_{\Ge^{(k)}}|(\phi_k)_t(y)f(xy)(T)|^2\,dy\\
&=\int_{\R^{k-1}_+}\frac{dT}{T}\int_0^{\infty}\frac{dt}{t}
\int_{\Ge}\int_{\Ge^{(k)}}|(\phi_k)_t(y)f(xy)(T)|^2\,dydx,\\
&=\int_{\R^{k-1}_+}\frac{dT}{T}\int_0^{\infty}\frac{dt}{t}<f_T\star(\Phi_k)_t,f_T\star(\Phi_k)_t>
=\|f\|_{L^2(\Ge,X_{k-1})}^2,
\end{align*}
where $f_T(x)=f(x)(T)$. 

Another property of $W_k$ that is needed is the following.  For every $\alpha$
\begin{equation}\label{Kestimate}
\|D^{\alpha}F_k(x)\|_{(X_{k-1},X_k)}\le C_{\alpha}|x|_k^{-Q-|\alpha|}.
\end{equation}
This follows readily from (\ref{fiestim}) specialized to $\phi_k$:
\[
 |D^{\alpha}\phi_k(x)|\le C_{\alpha}(1+|x|_k)^{-Q-1/2-|\alpha|}.
\]
As a bounded operator from $L^2(\Ge,X_{k-1})$ to $L^2(\Ge,X_k)$ satisfying (\ref{Kestimate}) is $W_k$  a vector-valued kernel of type $0$, and, by Theorem 6.20.b of Folland-Stein \cite{folland},  maps $L^p(\Ge,X_{k-1})$ into $L^p(\Ge,X_k)$ boundedly for every $1<p<\infty$.

This implies our assertion. In fact,
\[
G_{\Phi}(f)(x)=\|f\star F_1\star\dots\star F_d(x)\|_{X_d},
\]
and therefore
\[
 \|G_{\Phi}(f)\|_{L^p(\Ge)}=\|T_dT_{d-1}\dots T_1f\|_{L^p(\Ge,X_d)}
\le C\|f\|_{L^p(\Ge,X_0)}=C\|f\|_p. 
\]
\end{proof}

A word of comment on the symbol $\Phi_T$ would be appropriate here. The notation may suggest that the functions $\Phi_T$ are  dilates of a single function. They are not, but they have estimates of this form, which is our justification. The same applies to the symbol $K_T$ below. In the next section we are going to use the same notation for the ``real'' dilates of a function. We hope the reader will not get confused.

\section{The strong maximal function}

For a function $F$ on $\Ge$ and a $T\in\R^d_+$, let
\[
F_T(x)=F_{(t_1,t_2,\dots t_d)}(x)=t_1^{-Q_1}t_2^{-Q_2}\dots t_d^{-Q_d}F(t_1x_1,t_2x_2,\dots,t_dx_d).
\]

\textit{The strong maximal function} on $\Ge$ is defined by
\[
{\bf M}f(x)=\sup_{T\in\R_+^d}\int_{|y|\le1}|f(x(Ty)^{-1})\,dy=\sup_{T}|f\star(\chi_B)_T(x)|,
\]
where $\chi_B$ stands for the characteristic function of the unit ball $B=\{x\in\Ge:|x|\le1\}$, and $Ty=(t_1y_1,\dots,t_dy_d)$.  A theorem of Michael Christ asserts that for every $1<p<\infty$ there exists a constant $C>0$ such that
\[
\|{\bf M}f\|_p\le C\|f\|_p,
\qquad
f\in L^p(\Ge),
\]
that is, ${\bf M}$ is of $(p,p)$ type (see Christ \cite{christ}).

We shall need the following corollary to the Christ theorem. Let
\[
\gamma(t)=\min\{t,t^{-1}\},
\qquad
t>0.
\]
\begin{corollary}\label{mchrist}
Let 
\[
F(x)=\Pi_{j=1}^d\gamma(|x_j|)^a|x_j|^{-Q_j},
\qquad
x\neq0,
\]
for some $a>0$.Then the maximal fuction
\[
M_Ff(x)=\sup_{T\in\R_+^d}|f\star F_T(x)|
\]
is of $(p,p)$ type for $1<p<\infty$.
\end{corollary}

\begin{proof}
Let $B_j$ be the unit ball in $\Ge_j$ and let $|B_j|$ be the Lebesgue measure of $B_j$. Let $D=B_1\times\dots\times B_d$. Then for every simple positive function $h\le F$ of the form
\[
h(x)=\sum_{R}c_R\chi_{D}(R^{-1}x),
\qquad
R=(r_1,r_2,\dots,r_d)\in\R_+^d,
\]
we have
\[
h_T(x)=\sum_{R}c_Rr_1^{Q_1}r_2^{Q_2}\dots r_d^{Q_d}(\chi_D)_{RT}(x)
=\frac{C\|h\|_1}{|D|}(\chi_D)_{RT}(x),
\]
and therefore
\[
M_Ff(x)\le\frac{C\|F\|_1}{|D|}{\bf M}f(x),
\]
which completes the proof.
\end{proof}

\section{Flag kernels}

We keep the notation established in previous sections.
\begin{lemma}
Let 
\[
K_{T,S}=\widetilde{\Phi_{TS}}\star K\star{\Phi}_{T},
\qquad
T,S\in\R^d_+.
\]
Then $K_{T,S}\in\Fe(0)$ uniformly, and satisfy the estimates
\begin{equation}\label{F0}
|D^{\alpha}\widehat{K}_{T,S}(\xi)|\le C_{\alpha}\gamma(S)^{1/2}|\xi|_1^{-|\alpha_1|}\dots|\xi|_d^{-|\alpha_d|},
\end{equation}
where
\[
 \gamma(S)=\gamma(s_1)\gamma(s_2)\cdots\gamma(s_d).
\]
\end{lemma}

\begin{proof}
By the first part of Corollary \ref{unity}, $\Phi_T\in|\Fe|(1/2)$ with bounds uniformly proportional to $\gamma(T)^{1/2}$. Note that 
\[
\gamma(TS)\le\gamma(T)\cdot\gamma(S).
\]
Thus,  our assertion follows by  Proposition \ref{composition_flag}.
\end{proof}

We let
$$
K_T=K\star\Phi_T,
\qquad
T\in\R_+^d.
$$

\begin{lemma}\label{maxi}
For every $T$, $K_{T}$ is an integrable function, and the maximal operator
\begin{equation}\label{max}
K_{\Phi}^{\star}f(x)=\sup_{T}|f\star|\widetilde{K}_{T}|(x)|
\end{equation}
is of type $(p,p)$ for all $1<p<\infty$.
\end{lemma}
\begin{proof}
Observe that by Proposition \ref{l2bound_flag}, $K_{T}\in L^2(\Ge)$ so it is a function. Moreover,
by Corollary \ref{unity} and Proposition \ref{composition_flag}, it is a smooth away from $x_1=0$, and satisfies  
\[
|K_{T}(x)|\le C\gamma(T)^{1/2}
\gamma(|x|_1)^{1/2}|x|_1^{-Q_1}\dots\gamma(|x|_d)^{1/2}|x|_d^{-Q_d}
\]
uniformly in $T$ so that $K_{T}\le CF_T$, where $F_T$ is a dilate of
\[
F(x)=\gamma(|x|_1)^{1/2}|x|_1^{-Q_1}\dots\gamma(|x|_d)^{1/2}|x|_d^{-Q_d}.
\]
This shows that $K_{T}$ is integrable. The second part of our claim follows by Corollary \ref{mchrist} and the above.
\end{proof}

We turn to the main result of this paper. The reader may wish to compare the proof we give with that of  Theorem B  and the preceding lemma of Duoandicoetxea-Rubio de Francia \cite{duoandi}. 

\begin{theorem}\label{main}
Let $K$ be a flag kernel on $\Ge$. Then  the singular integral operator
\[
f\to f\star\widetilde{K},
\qquad
f\in\Schw(\Ge),
\]
extends uniquely to a bounded operator on $L^p(\Ge)$ for all  $1<p<\infty$.
\end{theorem}

\begin{proof}
Let $f,h\in\Schw(\Ge)$. We have
\begin{align*}\label{decomposition}
<f\star\widetilde{K},h>
&=\int_{\R_+^d}\frac{dS}{S}\int_{\R_+^d}\frac{dT}{T}
<f\star\Phi_{T},h\star\Phi_{TS}\star\widetilde{\Phi_{TS}}\star K\star\Phi_T>\\
&=\int_{\R_+^d}\frac{dS}{S}\int_{\R_+^d}\frac{dT}{T}
<f_T,h_{TS}\star K_{T,S}>,
\end{align*}
where
\[
f_{T}=f\star{\Phi}_{T},
\qquad
h_{TS}=h\star\Phi_{TS},
\qquad
K_{T,S}=\widetilde{\Phi_{TS}}\star K\star\Phi_T.
\]

We are going to estimate 
\[
<L_Sf,h>=\int_{\R_+^d}\frac{dT}{T}<f_{T},h_{TS}\star K_{T,S}>
\]
for a given $S$. Let us start with $L^2$-estimates. We have
\[
|<L_Sf,h>|
\le\left(\int_{\Ge}\int_{\R_+^d}|f_{T}(x)|^2\,\frac{dT}{T}dx\right)^{1/2}
\cdot\left(\int_{\Ge}\int_{\R_+^d}|h_{TS}\star K_{T,S}(x)|^2\,\frac{dT}{T}dx\right)^{1/2}.
\]

By (\ref{F0}) and Proposition \ref{l2bound_flag}, the operators $f\to f\star K_{T,S}$ are  bounded with norm estimates uniformly proportional to $\gamma(S)^{1/2}$ so that, by  Proposition \ref{paley},
\begin{align*}
|<L_Sf,h>|&\le C\gamma(S)^{1/2}\,\|G_{\Phi}(f)\|_2\,\|G_{\Phi}(h)\|_2\\
&\le C_1\gamma(S)^{1/2}\|f\|_2\|h\|_2,
\end{align*}
that is,
\begin{equation}\label{l2}
\|L_Sf\|_2\le C_1\gamma(S)^{1/2}\|f\|_2,
\qquad
f\in\Schw(\Ge).
\end{equation}
For $1<p<2$ and $f,h\in\Schw(\Ge)$,
\begin{align*}
|<L_Sf,h>|&
\le\int_{\Ge}\left(\int_{\R_+^d}|f_{T}(x)|^2\frac{dT}{T}\right)^{1/2}
\left(\int_{\R_+^d}|h_{TS}\star K_{T,S}(x)|^2\,\frac{dT}{T}\right)^{1/2}\,dx\\
&\le C_1\|G_{\Phi}(f)\|_p\left(\int_{\Ge}\left(\int_{\R_+^d}|h_{TS}\star K_{T,S}(x)|^2\,\frac{dT}{T}\right)^{q/2}dx\right)^{1/q}\\
&=C_2\|f\|_p\cdot\left\|\int_{\R_+^d}|h_{TS}\star K_{T,S}(\cdot)|^2\,\frac{dT}{T}\right\|_{q/2}^{1/2},
\end{align*}
where  $1/p+1/q=1$. Note that $q>2$. Thus, there exists a nonnegative function $u$ with $\|u\|_r=1$, where $2/q+1/r=1$, such that
\[
\left\|\int_{\R_+^d}|h_{TS}\star K_{T,S}(\cdot)|^2\,\frac{dT}{T}\right\|_{q/2}
=\int_{\Ge}\int_{\R_+^d}|h_{TS}\star K_{T,S}(x)|^2\,\frac{dT}{T}\cdot u(x)\, dx.
\]
Now,
\begin{align*}
h_{TS}\star K_{T,S}&=(h\star\Phi_{TS})\star(\widetilde{\Phi_{TS}}\star K\star\Phi_T)\\
&=(h\star\Phi_{TS}\star\widetilde{\Phi_{TS}})\star(K\star\Phi_T)=h'_{TS}\star K_T.
\end{align*}
Recall also that, by Lemma \ref{maxi}, $K_T$ are integrable functions. Therefore, by Lemma\nobreak\  \ref{maxi} again,
 \begin{align}\label{nomax}
\left\|\int_{\R_+^d}|h'_{TS}\star K_{T}(\cdot)|^2\,\frac{dT}{T}\right\|_{q/2}
&\le C_1\int_{\R_+^d}\int_{\Ge}|h'_{TS}|^2\star|K_{T}|(x)\cdot u(x)\, dx\frac{dT}{T}\notag\\
&\le C_2\int_{\Ge}\int_{\R^d}|h'_{TS}(x)|^2\,\frac{dT}{T}\cdot K_{\Phi}^{\star}u(x)\,dx\notag\\
&\le C_3\|G_{\Phi}(h)\|_q^2\cdot\|K^{\star}_{\Phi}u\|_r\le C_4\|h\|_q^2,
\end{align}
where we have used the estimate
 \begin{align*}
|h'_{TS}\star K_{T}(x)|^2&\le\left(\int_{\Ge}|h'_{TS}(xy^{-1})|\cdot|K_{T}(y)|^{1/2}\cdot|K_{T}(y)|^{1/2}\,dy\right)^2\\
&\le\int_{\Ge}|h'_{TS}|^2(xy^{-1})\cdot|K_{T}|(y)\,dy\cdot\int_{\Ge}|K_{T}(y)|\,dy\\
&\le C|h'_{TS}|^2\star|K_{T}|(x),
\end{align*}
the integrals 
\[
\int_{\Ge}|K_{T}(x)|\,dx\le C
\] 
being uniformly bounded, as can be seen from the proof of Lemma \ref{maxi}. Therefore, 
\begin{equation}\label{lp}
\|L_Sf\|_p\le C_1\|f\|_p.
\end{equation}

\medskip
Now, by interpolating between (\ref{l2}) and (\ref{lp}), we get
\[
\|L_Sf\|_p\le C_2\gamma(S)^{\e_p}\|f\|_p,
\]
where $\e_p>0$ depends only on $p$, and, finally, 
\[
\|f\star\widetilde{K}\|_p\le C_3\left(\int_{\R_+^d}\gamma(S)^{\e_p}\,\frac{dS}{S}\right)\cdot\|f\|_p
=C_4\|f\|_p,
\]
which proves our case for $1<p\le2$. The result for $2<p<\infty$ follows by duality.
\end{proof}

\section*{Acknowledgements}
I wish to extend thanks to Alexander Nagel, Fulvio Ricci, and Elias M. Stein for their interest in my work and an inspiring conversation. I am also indebted to Fran{\c c}ois Piquard for pointing out an editorial omission in the initial version of the manuscript.


\newpage
\begin{thebibliography}{99}
\bibitem{christ}
{M. Christ}, {The strong maximal function on a nilpotent group}, {\em Trans. Am. Math. Soc.},
 331 (1992), 1-13,

\bibitem{duoandi}
{J. Duoandikoetxea and Rubio de Francia},
{Maximal and singular integral operators via Fourier transform estimates},
{\em Inv. math.}, 84 (1986), 541-561,

\bibitem{folland}
{G.B. Folland and E.M. Stein}, 
{\em Hardy spaces on homogeneous groups},
Princeton University Press, Princeton 1982,

\bibitem{inventiones1986}
{P. G{\l}owacki}, {Stable semigroups of measures as commutative approximate identities on non-graded homogeneous groups}, {\em Inv. Math.}, 83 (1986), 557-582,

\bibitem{arkiv2007}
{P. G{\l}owacki}, {The Melin calculus for general homogeneous groups}, {\em Arkiv f\"or matematik}, 45 (2007), 31-48,

\bibitem{colloquium2010}
{P. G{\l}owacki}, {Composition and $L^2$-boundedness of flag kernels},
{\em Colloq. Math.}, 118 (2010), 581-585,

\bibitem{studia2010}
{P. G{\l}owacki}, {Flag kernels of arbitrary order}, {\em submitted},

\bibitem{nagel}
{A. Nagel, F. Ricci, and E.M. Stein}, {Singular integrals with flag kernels and analysis on quadratic CR manifolds}, {\em J. Func. Analysis} 181, 29-118 (2001).

\bibitem{stein}
{E. M. Stein}, {\em Algebras of operators},
http://www.math.ucla.edu/dls/2009/stein.shtml,

\bibitem{yosida}
{K. Yosida}
{\em Functional analysis},
Springer-Verlag, Berlin-Heidelberg-New York 1980.

\end{thebibliography}
\end{document}